\documentclass[11pt,twoside,reqno]{amsart}

\usepackage{amssymb, latexsym}
\usepackage[italian,english]{babel}
\usepackage{amsmath,amsfonts,amsthm}
\usepackage{paralist}
\usepackage[top=2.50cm, bottom=2.50cm, left=2.50cm, right=2.50cm]{geometry}

\usepackage{color}

\numberwithin{equation}{section}

\usepackage{xspace}
\usepackage[bookmarksnumbered,colorlinks]{hyperref}
\usepackage{graphics}
\def\bb#1\eb{\textcolor{blue}
{#1}} %
\def\br#1\er{\textcolor{red}
{#1}} %
\def\bv#1\ev{\textcolor{green}
{#1}} %
\def\bc#1\ec{\textcolor{cyan}
{#1}} %

\usepackage{graphics}

\def\Xint#1{\mathchoice
  {\XXint\displaystyle\textstyle{#1}}%
  {\XXint\textstyle\scriptstyle{#1}}%
  {\XXint\scriptstyle\scriptscriptstyle{#1}}%
  {\XXint\scriptscriptstyle\scriptscriptstyle{#1}}%
  \!\int}
\def\XXint#1#2#3{{\setbox0=\hbox{$#1{#2#3}{\int}$}
  \vcenter{\hbox{$#2#3$}}\kern-.5\wd0}}
\def\-int{\Xint -}

\newcommand{\R}{\mathbb{R}}
\newcommand{\N}{\mathbb{N}}

\DeclareMathOperator{\supp}{supp}
\DeclareMathOperator{\B}{\mathcal{B}}
\DeclareMathOperator{\K}{\mathcal{K}}
\DeclareMathOperator{\X}{\mathbb{X}}
\DeclareMathOperator{\Q}{\mathcal{Q}}

\newtheorem{prop}{Proposition}[section]
\newtheorem{lem}{Lemma}[section]
\newtheorem{thm}{Theorem}[section]

\begin{document}
\title[]{Sign-changing solutions for a class of Schr\"odinger equations with vanishing potentials}

\author[V. Ambrosio]{Vincenzo Ambrosio}
\address{
Vincenzo Ambrosio\hfill\break\indent 
Dipartimento di Scienze Pure e Applicate (DiSPeA),\hfill\break\indent
Universit\`a degli Studi di Urbino `Carlo Bo'\hfill\break\indent
Piazza della Repubblica, 13\hfill\break\indent
61029 Urbino (Pesaro e Urbino, Italy)}
\email{vincenzo.ambrosio@uniurb.it}

\author[T. Isernia]{Teresa Isernia}
\address{
Teresa Isernia\hfill\break\indent
Dipartimento di Matematica e Applicazioni `Renato Caccioppoli' \hfill\break\indent
Universit\`a degli Studi di Napoli `Federico II' \hfill\break\indent
Via Cintia 1, 80126 Napoli, 
Italy}
\email{teresa.isernia@unina.it}

\date{}


\begin{abstract}
In this paper we consider a class of fractional Schr\"odinger equations with potentials vanishing at infinity. By using a minimization argument and a quantitative Deformation Lemma, we prove the existence of a sign-changing solution.
\end{abstract}

\maketitle

\section{Introduction}

\noindent
In the past years there has been a considerable amount of research related to the existence of nontrivial solutions for Schr\"odinger-type equations
\begin{equation}\label{SE}
\left\{
\begin{array}{ll}
-\Delta u + V(x)u = K(x)f(u) \mbox{ in } \R^{N}\\
u\in \mathcal{D}^{1,2}(\R^{N})
\end{array}
\right.
\end{equation}
where $V:\R^{N}\rightarrow \R$ and $K:\R^{N}\rightarrow \R$ are positive and continuous functions, and $f:\R\rightarrow \R$ is a nonlinearity satisfying suitable growth assumptions in the origin and at infinity. An important class of problems associated to (\ref{SE}) is the so called zero mass case, which occurs when the potential $V(x)$ vanishes at infinity. Such class of problems has been investigated by many authors by using several variational methods; see for instance \cite{AS1, AFM, AW, BGM2, BL1, BVS}.

\noindent
Recently, the study of nonlinear equations involving the fractional Laplacian has gained tremendous popularity due to their intriguing analytic structure and in view of several applications in different subjects, such as, optimization, finance, anomalous diffusion, phase transition, flame propagation, minimal surface.
The literature on fractional and non-local operators of elliptic type and their applications is quite large, for example, we refer the interested reader to \cite{A0, BCDS, Cab, CS, CG, FV, MolSer, ROS, sv1, sv2} and references therein. For the basic properties of fractional Sobolev spaces with applications to partial differential equations, we refer the reader to \cite{DPV, MRS} and references therein. \\
Motivated by the interest shared by the mathematical community in this topic, the purpose of this paper is to study sign-changing (or nodal) solutions for the following class of fractional equations
\begin{equation}\label{P}
(-\Delta)^{\alpha} u + V(x)u = K(x) f(u) \, \mbox{ in } \R^{N},
\end{equation}
where $\alpha\in (0,1)$ and $N>2\alpha$, $(-\Delta)^{\alpha}$ is the fractional Laplacian which may be defined for a function $u$ belonging to the Schwartz space $ \mathcal{S}(\R^{N})$ of rapidly decaying functions as
$$
(-\Delta)^{\alpha}u(x)=C_{N,\alpha} \,P. V. \int_{\R^{N}} \frac{u(x)-u(y)}{|x-y|^{N+2\alpha}} dy,  \quad x\in \R^{N},
$$
where P.V. stands for the Cauchy principal value and $C_{N,\alpha}$ is a normalizing constant \cite{DPV}. \\
Here, we assume that $V, K : \R^{N} \rightarrow \R$ are continuous functions verifying appropriate hypotheses. 
More precisely, as in \cite{AS1}, we say $(V, K)\in \K$ if the following conditions hold:
\begin{compactenum}[($VK_1$)]
\item $V(x), K(x)>0$ for all $x \in \R^{N}$ and $K \in L^{\infty}(\R^{N})$; 
\item If $\{A_{n}\}\subset \R^{N}$ is a sequence of Borel sets such that the Lebesgue measure $|A_{n}|\leq R$, for all $n\in \N$ and some $R>0$, then 
\begin{equation*}
\lim_{r\rightarrow +\infty} \int_{A_{n}\cap \B_{r}^{c}(0)} K(x) \, dx =0, \, \mbox{ uniformly in } n\in \N.  
\end{equation*}
\end{compactenum}
Furthermore, one of the below conditions occurs
\begin{compactenum}[($VK_3$)]
\item $\displaystyle{\frac{K}{V}\in L^{\infty}(\R^{N}) }$
\end{compactenum}
or 
\begin{compactenum}[($VK_4$)]
\item there exists $m\in (2, 2^{*}_{\alpha})$ such that 
\begin{equation*}
\frac{K(x)}{V(x)^{\frac{2^{*}_{\alpha}-m}{2^{*}_{\alpha}-2}}} \rightarrow 0 \, \mbox{ as } |x|\rightarrow +\infty, 
\end{equation*}
where $\displaystyle{2^{*}_{\alpha}= \frac{2N}{N-2\alpha}}$ is the fractional critical exponent. 
\end{compactenum}

\noindent
Concerning the nonlinearity $f: \R\rightarrow \R$, we assume that $f$ is a $C^{1}$-function and fulfills the following growth conditions: 
\begin{compactenum}[($f_1$)]
\item $\displaystyle{\lim_{|t|\rightarrow 0} \frac{f(t)}{|t|}=0} \mbox{ if } (VK_3) \mbox{ holds }$
\end{compactenum}
or 
\begin{compactenum}[($\tilde{f}_1$)]
\item $\displaystyle{\lim_{|t|\rightarrow 0} \frac{f(t)}{|t|^{m-1}}=0} \mbox{ if } (VK_4)$ holds with $m \in (2, 2^{*}_{\alpha})$ defined in $(VK_4)$. 

\item [($f_{2}$)] $f$ has a quasicritical growth at infinity, namely
\begin{equation*}
\lim_{|t|\rightarrow +\infty} \frac{f(t)}{|t|^{2^{*}_{\alpha}-1}}=0,
\end{equation*}

\item [($f_{3}$)] There exists $\theta \in (2, 2^{*}_{\alpha})$ so that 
\begin{equation*}
0< \theta F(t) = \theta \int_{0}^{t} f(\tau) \, d\tau \leq f(t) \, t \, \mbox{ for all } |t|> 0, 
\end{equation*}

\item [($f_{4}$)] The map $f$ and its derivative $f'$ satisfy 
\begin{equation*}
f'(t) < \frac{f(t)}{t} \, \mbox{ for all } t\neq 0. 
\end{equation*}
\end{compactenum}

\noindent
Let us observe that by $(f_4)$ follows that $\displaystyle{t \mapsto \frac{f(t)}{|t|}}$ is strictly increasing for all $|t|>0$. Moreover,
\begin{align}\begin{split} \label{4.23} 
t\mapsto \frac{1}{2} f(t)t - F(t) &\mbox{ is strictly increasing for every } t>0 \\
& \mbox{ is strictly decreasing for every } t<0
\end{split}\end{align}
and, in particular 
\begin{equation}\label{vince}
\frac{t^{2}}{2} f'(t) - \frac{t}{2} f(t)>0 \, \mbox{ for all } \, t\neq 0.
\end{equation}

\noindent
Equation (\ref{P}) appears in a lot of studies, for instance, when we look for standing wave solutions $\psi(x,t)=u(x) e^{-\imath \omega t}$ to the following fractional Schr\"odinger equation 
$$
\imath \hbar \frac{\partial\psi}{\partial t}=\hbar^{2}(-\Delta)^{\alpha}\psi+W(x) \psi - f(x, \psi), \quad x\in \R^{N},
$$
where $\hbar$ is the Planck's constant, $W:\R^{N}\rightarrow \R$ is an external potential and $f$ is a suitable nonlinearity.
This equation plays an important role in fractional quantum mechanic, and was introduced by Laskin \cite{Laskin1, Laskin2} through expanding the Feynman path integral from the Brownian-like to the L\'evy-like quantum mechanical paths. \\
Lately the study of fractional Schr\"odinger equations has attracted the attention of many mathematicians; see for instance \cite{A1, A2, CZ, CTN1, DiPV, FQT, FLS, MR, PXZ, Secchi1} and references therein. \\
In spite of the fact that there are many papers dealing with existence and multiplicity of solutions of fractional Schr\"odinger equations in $\R^{N}$, to our knowledge there are no papers dealing with the existence of sign-changing solutions for fractional Schr\"odinger equations with potentials vanishing at infinity, and here we would like to go further in this direction. \\

\noindent
The main result of this paper is the following:
\begin{thm}\label{thm}\label{thm1}
Suppose that $(V, K)\in \K$ and $f\in C^{1}(\R, \R)$ verifies $(f_1)$ or $(\tilde{f}_1)$, and $(f_2)$-$(f_4)$. Then, the problem (\ref{P}) admits a least energy sign-changing weak solution. 
\end{thm}

\noindent
For weak solution to (\ref{P}) we mean a function $u\in \X$ such that
\begin{equation*}
\iint_{\R^{2N}} \frac{(u(x)- u(y)) (\varphi(x)-\varphi(y))}{|x-y|^{N+2\alpha}} \, dxdy + \int_{\R^{N}} V(x) u(x) \varphi(x) \, dx = \int_{\R^{N}} K(x) f(u) \varphi(x) \, dx
\end{equation*}
for all $\varphi \in \X$, where 
\begin{equation*}
\X= \left\{u \in \mathcal{D}^{\alpha,2}(\R^{N}) \, : \, \int_{\R^{N}} V(x) |u|^{2} \, dx <+\infty \right\}. 
\end{equation*}
The proof of Theorem \ref{thm1} is obtained by adapting some arguments developed in \cite{AS2, BF}.
More precisely, we minimize the Euler-Lagrange functional associated to \eqref{P}, that is 
\begin{equation*}
J(u):= \frac{1}{2}\iint_{\R^{2N}} \frac{|u(x)- u(y)|^{2}}{|x-y|^{N+2\alpha}} \, dxdy+ \frac{1}{2}\int_{\R^{N}} V(x) |u|^{2} \, dx - \int_{\R^{N}} K(x) F(u) \, dx,
\end{equation*} 
on the nodal set 
\begin{equation*}
\mathcal{M}:= \{ w\in \mathcal{N} : w^{+}\neq 0, w^{-}\neq 0,  \langle J'(w), w^{+} \rangle = 0 =  \langle J'(w), w^{-} \rangle\},
\end{equation*}
where 
\begin{equation*}
\mathcal{N}:= \{ u\in \X \setminus \{0\} : \langle J'(u), u \rangle =0\}.
\end{equation*}
Then we prove that the minimum is achieved and, by using a suitable variant of the quantitative deformation Lemma, we show that it is a critical point of $J$. Clearly, due to the presence of the nonlocal term $\iint_{\R^{2N}} \frac{|u(x)- u(y)|^{2}}{|x-y|^{N+2\alpha}} \, dxdy$, the Euler-Lagrange functional $J$ does no longer satisfy the decompositions 
\begin{align*}
&J(u)= J(u^{+})+ J(u^{-}) \\
&\langle J'(u), u^{\pm} \rangle= \langle J'(u^{\pm}), u^{\pm} \rangle,
\end{align*}
which were very useful to get sign-changing solutions to \eqref{SE}; see for instance \cite{BF, BLW, BW, BWW, CCN}. Therefore, in order to prove the existence of a sign-changing solution to (\ref{P}), a more accurate investigation is needed in our setting.
\smallskip 

\noindent
The paper is organized as follows. In Section \ref{PR} we present the variational setting of the problem and we provide some compactness results which will be useful for the next sections. In Section \ref{TL} we give some technical lemmas used in the proof of the main result. Finally, in Section \ref{MRT} we prove Theorem \ref{thm1} by minimization arguments and a variant of Deformation Lemma.

\section{Preliminary results}\label{PR}

\noindent
Firstly we recall some basic notation and facts which will be used in the sequel of the paper. \\
We denote by $\mathcal{D}^{\alpha,2}(\R^{N})$ the closure of functions $C^{\infty}_{c}(\R^{N})$ with respect to the so called Gagliardo seminorm
\begin{equation*}
[u]^{2}= \iint_{\R^{2N}} \frac{|u(x)- u(y)|^{2}}{|x-y|^{N+2\alpha}} \, dxdy. 
\end{equation*}
In order to prove that problem (\ref{P}) has a variational structure, let us introduce the Hilbert space
\begin{equation*}
\X= \left\{u \in \mathcal{D}^{\alpha,2}(\R^{N})\, : \, \int_{\R^{N}} V(x) |u|^{2} \, dx <+\infty \right\}
\end{equation*}
endowed with the norm 
\begin{equation*}
\|u\|^{2}= [u]^{2} + \int_{\R^{N}} V(x)|u|^{2} dx. 
\end{equation*}

\noindent
Let $q \in \R$ such that $q\geq 1$, and let us define the weighted Lebesgue space 
\begin{equation*}
L_{K}^{q} (\R^{N}) = \left \{ u: \R^{N} \rightarrow \R \, \mbox{ measurable and } \int_{\R^{N}} K(x) |u|^{q} \, dx < \infty \right \}
\end{equation*}
endowed with the norm
\begin{equation*}
\|\cdot \|_{L_{K}^{q} (\R^{N})}= \left( \int_{\R^{N}} K(x) |u|^{q} \, dx \right)^{\frac{1}{q}}. 
\end{equation*}

\noindent
Now we prove the following continuous and compactness results, whose proofs can be obtained adapting the arguments in \cite{AS1}. For the reader's convenience we give the proofs.

\begin{lem}\label{cont}
Assume that $(V, K)\in \mathcal{K}$. Then $\X$ is continuously embedded in $L^{q}_{K}(\R^{N})$ for all $q\in [2, 2^{*}_{\alpha}]$ if $(VK_3)$ holds. Moreover, $\X$ is continuously embedded in $L^{m}_{K}(\R^{N})$ if $(VK_4)$ holds. 
\end{lem}

\begin{proof}
Assume that $(VK_3)$ is true. The proof is trivial if $q=2$ or $q=2^{*}_{\alpha}$. \\
Fix $q\in (2, 2^{*}_{\alpha})$ and let $\displaystyle{\lambda= \frac{2^{*}_{\alpha}-q}{2^{*}_{\alpha}-2}}$. We can observe that $q$ can be written as $q= 2\lambda + (1-\lambda)2^{*}_{\alpha}$. Then we have
\begin{align*}
\int_{\R^{N}} K(x)|u|^{q} \, dx &= \int_{\R^{N}} K(x) |u|^{2\lambda} |u|^{(1-\lambda)2^{*}_{\alpha}} \, dx \\
&\leq \left( \int_{\R^{N}} |K(x)|^{\frac{1}{\lambda}}|u|^{2}\, dx \right)^{\lambda}  \left( \int_{\R^{N}}|u|^{2^{*}_{\alpha}}\, dx \right)^{1-\lambda} \\
&\leq \left(\sup_{x\in \R^{N}} \frac{|K(x)|}{|V(x)|^{\lambda}} \right) \left( \int_{\R^{N}} V(x) |u|^{2}\, dx \right)^{\lambda}  \left( \int_{\R^{N}}|u|^{2^{*}_{\alpha}}\, dx \right)^{1-\lambda} \\
&\leq C \left(\sup_{x\in \R^{N}} \frac{|K(x)|}{|V(x)|^{\lambda}} \right) \left( \int_{\R^{N}} V(x) |u|^{2}\, dx \right)^{\lambda}  \left( \iint_{\R^{2N}} \frac{|u(x)- u(y)|^{2}}{|x-y|^{N+2\alpha}}\, dxdy \right)^{\frac{(1-\lambda)2^{*}_{\alpha}}{2}} \\
&\leq C \left(\sup_{x\in \R^{N}} \frac{|K(x)|}{|V(x)|^{\lambda}} \right) \|u\|^{2\left[\lambda +\frac{(1-\lambda)2^{*}_{\alpha}}{2}\right]}\\
&= C \left(\sup_{x\in \R^{N}} \frac{|K(x)|}{|V(x)|^{\lambda}} \right)  \|u\|^{q}. 
\end{align*}
Taking into account $K\in L^{\infty}(\R^{N})$ and that $(VK_3)$ holds true, we conclude that
\begin{equation*}
\|u\|_{L^{q}_{K}(\R^{N})}\leq C \|u\|. 
\end{equation*}
Now, we suppose that $(VK_4)$ is true. Denoting $\displaystyle{\lambda_{0}= \frac{2^{*}_{\alpha}-m}{2^{*}_{\alpha}-2}}$, we can see that $m$ can be written as $m= 2\lambda_{0}+ (1- \lambda_{0})2^{*}_{\alpha}$. As above, we have
\begin{align*}
\int_{\R^{N}} K(x)|u|^{m} \, dx &= \int_{\R^{N}} K(x) |u|^{2\lambda_{0}} |u|^{(1-\lambda_{0})2^{*}_{\alpha}} \, dx \\
&\leq \left( \int_{\R^{N}} |K(x)|^{\frac{1}{\lambda_{0}}}|u|^{2}\, dx \right)^{\lambda_{0}}  \left( \int_{\R^{N}}|u|^{2^{*}_{\alpha}}\, dx \right)^{1-\lambda_{0}} \\
&\leq \left(\sup_{x\in \R^{N}} \frac{|K(x)|}{|V(x)|^{\lambda_{0}}} \right) \left( \int_{\R^{N}} V(x) |u|^{2}\, dx \right)^{\lambda_{0}}  \left( \int_{\R^{N}}|u|^{2^{*}_{\alpha}}\, dx \right)^{1-\lambda_{0}} \\
&\leq C \left(\sup_{x\in \R^{N}} \frac{|K(x)|}{|V(x)|^{\lambda}} \right)  \|u\|^{m}. 
\end{align*}
Since $\displaystyle{\frac{K(x)}{V(x)^{\frac{2^{*}_{\alpha}-m}{2^{*}_{\alpha}-2}}}\in L^{\infty}(\R^{N})}$, we can infer that
\begin{equation*}
\|u\|_{L^{m}_{K}(\R^{N})}\leq C \|u\|. 
\end{equation*}
This complete the proof of the Lemma \ref{cont}. 
\end{proof}

\begin{prop}\label{prop2.2}
Assume $(V, K)\in \K$. The following facts hold: 
\begin{compactenum}[\it{(1)}]
\item [\it{(1)}] $\X$ is compactly embedded into $L_{K}^{q}(\R^{N})$ for all $q \in (2, 2^{*}_{\alpha})$ if $(VK_3)$ holds; 
\item [\it{(2)}] $\X$ is compactly embedded into $L_{K}^{m}(\R^{N})$ if $(VK_4)$ holds.
\end{compactenum}
\end{prop}

\begin{proof}
$\it{(1)}$ Assume that $(VK_3)$ holds. Fix $q \in (2, 2^{*}_{\alpha})$ and let $\varepsilon >0$. Then, there exist $0<t_{0}<t_{1}$ and a positive constant $C$ such that 
\begin{equation*}
K(x) |t|^{q} \leq \varepsilon \, C \left[V(x)|t|^{2} + |t|^{2^{*}_{\alpha}} \right] + C\, K(x) \chi_{[t_{0}, t_{1}]}(|t|) |t|^{2^{*}_{\alpha}}, \, \mbox{ for all } t\in \R. 
\end{equation*}
Integrating over $\B_{r}^{c}(0)$ we have, for all $u\in \X$ and $r>0$, 
\begin{align}\begin{split}\label{2.5}
\int_{\B_{r}^{c}(0)} K(x) |u|^{q} \, dx &\leq \varepsilon \, C \int_{\B_{r}^{c}(0)} \left[V(x) |u|^{2} + |u|^{2^{*}_{\alpha}} \right]\, dx + C t_{1}^{2^{*}_{\alpha}} \int_{A\cap \B_{r}^{c}(0)} K(x)\, dx \\
&=: \varepsilon \, C \mathcal{Q}(u) + C t_{1}^{2^{*}_{\alpha}} \int_{A\cap \B_{r}^{c}(0)} K(x)\, dx 
\end{split}\end{align}
where we set
\begin{equation*}
\mathcal{Q}(u):=\int_{\B_{r}^{c}(0)} \left[V(x) |u|^{2} + |u|^{2^{*}_{\alpha}} \right]\, dx \, \mbox{ and } \, A=\left\{ x\in \R^{N} : t_{0} \leq |u(x)| \leq t_{1}\right\}. 
\end{equation*}

\noindent
Now, if $\{u_{n}\}\subset \X$ is a sequence such that $u_{n}\rightharpoonup u$ in $\X$, then there is $M>0$ such that
\begin{equation}\label{ter}
\|u_{n}\|^{2}\leq M \, \mbox{ and }\, \int_{\R^{N}} |u_{n}|^{2^{*}_{\alpha}}\, dx \leq M \quad \forall n\in \N.  
\end{equation}
This implies that $\{\Q(u_n)\}$ is bounded from above by a positive constant. Let us denote by $A_{n}=\left\{ x\in \R^{N} : t_{0}\leq |u_n|\leq t_{1} \right \}$. By (\ref{ter}) we deduce 
\begin{equation*}
t_{0}^{2^{*}_{\alpha}} m(A_{n}) \leq \int_{A_{n}} |u_{n}|^{2^{*}_{\alpha}} dx \leq M \quad \forall n\in \N, 
\end{equation*}
which implies that $\sup_{n\in \N} |m(A_{n})|<+\infty$. Therefore, from $(VK_2)$ there exists a positive radius $r$ large enough such that 
\begin{equation}\label{new}
\int_{A_{n}\cap \B_{r}^{c}(0)} K(x) \, dx < \frac{\varepsilon}{t_{1}^{2^{*}_{\alpha}}} \, \mbox{ for all } n\in \N.
\end{equation}
Putting together (\ref{2.5}) and (\ref{new}) we obtain
\begin{align}\label{2.6}
\int_{\B_{r}^{c}(0)} K(x) |u_{n}|^{q} \, dx &\leq \varepsilon \, C \,M+ C\,t_{1}^{2^{*}_{\alpha}} \int_{A_{n}\cap \B_{r}^{c}(0)} K(x)\, dx  \nonumber \\
&\leq (C\,M +C) \varepsilon \, \mbox{ for all } n\in \N. 
\end{align}
Recalling that $q \in (2, 2^{*}_{\alpha})$ and that $K$ is a continuous function, by Sobolev embedding follows that
\begin{equation}\label{2.7}
\lim_{n\rightarrow \infty} \int_{\B_{r}(0)} K(x) |u_{n}|^{q} dx = \int_{\B_{r}(0)} K(x) |u|^{q} dx. 
\end{equation}
By (\ref{2.6}) for $\varepsilon>0$ small enough and (\ref{2.7}) it holds
\begin{equation*}
\lim_{n\rightarrow \infty} \int_{\R^{N}} K(x) |u_{n}|^{q} dx = \int_{\R^{N}} K(x) |u|^{q} dx
\end{equation*}
from which we conclude that
\begin{equation*}
u_{n} \rightarrow u \mbox{ in } L_{K}^{q}(\R^{N}), \mbox{ for every } q\in (2, 2^{*}_{\alpha}). 
\end{equation*}
\smallskip

\noindent
$\it{(2)}$ Let us suppose that $(VK_4)$ is true. Then, we can see that for each $x\in \R^{N}$ fixed, the function
\begin{equation*}
g(t)= V(x) t^{2-m} + t^{2^{*}_{\alpha}-m}, \mbox{ for every } t>0, 
\end{equation*}
has $C_{m} V(x)^{\frac{2^{*}_{\alpha}-m}{2^{*}_{\alpha}-2}}$ as its minimum value, where $\displaystyle{C_{m}=\left(\frac{2^{*}_{\alpha}-2}{2^{*}_{\alpha}-m}\right)\left(\frac{m-2}{2^{*}_{\alpha}-2}\right)^{\frac{2-m}{2^{*}_{\alpha}-2}}}$. 
Hence 
\begin{equation*}
C_{m} V(x)^{\frac{2^{*}_{\alpha}-m}{2^{*}_{\alpha}-2}} \leq V(x) t^{2-m} + t^{2^{*}_{\alpha}-m}, \mbox{ for every } x\in \R^{N} \mbox{ and } t>0. 
\end{equation*}
Combining the last inequality with $(VK_4)$, for any $\varepsilon> 0$, we can find a positive radius $r$ sufficiently large such that
\begin{equation*}
K(x)|t|^{m} \leq \varepsilon C'_{m} \left[V(x)|t|^{2} + |t|^{2^{*}_{\alpha}} \right], \mbox{ for every } t\in \R \mbox{ and } |x|\geq r,
\end{equation*}
where $C'_{m}$ is the inverse of $C_{m}$, and integrating over $\B_{r}^{c}(0)$ we get
\begin{equation}\label{ee}
\int_{\B_{r}^{c}(0)} K(x)|u|^{m} \, dx \leq \varepsilon C'_{m} \left[\|u\|^{2} + \|u\|_{L^{2^{*}_{\alpha}}(\R^{N})}^{2^{*}_{\alpha}}\right] \mbox{ for all } u\in \X.  
\end{equation}
If $\{u_{n}\}\subset \X$ is a sequence such that $u_{n} \rightharpoonup u$ in $\X$, by (\ref{ee}) we deduce that
\begin{equation}\label{eee}
\int_{\B_{r}^{c}(0)} K(x)|u|^{m} \, dx \leq \varepsilon C''_{m} \mbox{ for all } n\in \N. 
\end{equation}
Once that $m\in (2, 2^{*}_{\alpha})$ and $K$ is a continuous function, it follows from Sobolev embedding that 
\begin{equation}\label{2.7e}
\lim_{n\rightarrow \infty} \int_{\B_{r}(0)} K(x) |u_{n}|^{m} dx = \int_{\B_{r}(0)} K(x) |u|^{m} dx. 
\end{equation}
Then, (\ref{eee}) and (\ref{2.7e}) yield
\begin{equation*}
\lim_{n\rightarrow \infty} \int_{\R^{N}} K(x) |u_{n}|^{m} dx = \int_{\R^{N}} K(x) |u|^{m} dx, 
\end{equation*}
from which we can infer that
\begin{equation*}
u_{n} \rightarrow u \mbox{ in } L_{K}^{m}(\R^{N}), \mbox{ for every } m \in (2, 2^{*}_{\alpha}). 
\end{equation*}
\end{proof}

\noindent
The following lemma is a compactness result related to the nonlinear term. 
\begin{lem}\label{lem2.1}
Assume $(V, K)\in \K$ and $f$ satisfies $(f_1)$-$(f_2)$ or $(\tilde{f}_1)$-$(f_2)$. Let $\{u_{n}\}$ be a sequence such that $u_{n}\rightharpoonup u$ in $\X$, then, up to a subsequence, 
\begin{equation*}
\lim_{n\rightarrow \infty} \int_{\R^{N}} K(x) F(u_{n}) \, dx = \int_{\R^{N}} K(x) F(u) \, dx
\end{equation*}
and 
\begin{equation*}
\lim_{n\rightarrow \infty} \int_{\R^{N}} K(x) f(u_{n}) u_{n} \, dx = \int_{\R^{N}} K(x) f(u) u \, dx.
\end{equation*}
\end{lem}

\begin{proof}
Assume that $(VK_3)$ holds. From $(f_1)$-$(f_2)$, fixed $q \in (2, 2^{*}_{\alpha})$ and given $\varepsilon>0$, there exists $C>0$ such that 
\begin{equation}\label{2.8}
|K(x)f(t)t|\leq \varepsilon C \left[V(x) |t|^{2} + |t|^{2^{*}_{\alpha}} \right] + C K(x) |t|^{q}, \mbox{ for all } t\in \R. 
\end{equation}
From Proposition \ref{prop2.2} we know that
\begin{equation*}
\lim_{n\rightarrow \infty} \int_{\R^{N}} K(x)|u_{n}|^{q} dx = \int_{\R^{N}} K(x)|u|^{q} dx,
\end{equation*}
so there exists $r>0$ such that
\begin{equation}\label{2.9}
\int_{\B_{r}^{c}(0)} K(x) |u_{n}|^{q} dx <\varepsilon, \mbox{ for all } n\in \N. 
\end{equation}
Since $\{u_{n}\}\subset \X$ is bounded, there exists a positive constant $C'$ such that 
\begin{equation}\label{aa}
\int_{\R^{N}} V(x)|u_{n}|^{2} dx \leq C' \, \mbox{ and } \, \int_{\R^{N}} |u_{n}|^{2^{*}_{\alpha}} dx \leq C', \mbox{ for all } n\in \N. 
\end{equation}
Taking into account (\ref{2.8}), (\ref{2.9}), and (\ref{aa}) we have
\begin{equation*}
\int_{\B_{r}^{c}(0)} K(x) |u_{n}|^{q} dx< (2CC' + 1) \varepsilon, \mbox{ for all } n\in \N.   
\end{equation*}

\noindent
Assume that $(VK_4)$ holds. Similarly to the second part of Proposition \ref{prop2.2}, given $\varepsilon>0$ sufficiently small, there exists $r>0$ large enough such that
\begin{equation*}
K(x)\leq \varepsilon C'_{m} \left[V(x) |t|^{2-m} + |t|^{2^{*}_{\alpha}-m}\right], \mbox{ for every } |t|>0 \mbox{ and } |x|>r.  
\end{equation*}
Consequently, for all $|t|>0$ and $|x|>r$
\begin{equation*}
K(x) |f(t)t|\leq \varepsilon C'_{m} \left[V(x) |f(t)t| |t|^{2-m} + |f(t)t| |t|^{2^{*}_{\alpha}-m}\right].  
\end{equation*}
From $(\tilde{f}_{1})$ and $(f_2)$, there exist $C, t_{0}, t_{1}>0$ satisfying 
\begin{equation*}
K(x) |f(t)t|\leq \varepsilon C \left[V(x) |t|^{2} + |t|^{2^{*}_{\alpha}}\right], \mbox{ for every } t\in I \mbox{ and } |x|>r  
\end{equation*}
where $I=\{t \in \R : |t|<t_{0} \mbox{ or } |t|>t_{1}\}$. Therefore, for every $u\in \X$, setting 
\begin{equation*}
\Q(u)= \int_{\R^{N}} V(x)|u|^{2} \, dx + \int_{\R^{N}} |u|^{2^{*}_{\alpha}}\, dx
\end{equation*}
and $A=\{x\in \R^{N} : t_{0}\leq |u(x)|\leq t_{1}\}$, the following estimate holds
\begin{equation*}
\int_{\B_{r}^{c}(0)} K(x) f(u)u \, dx \leq \varepsilon C \Q(u) + C \int_{A\cap \B_{r}^{c}(0)} K(x)\, dx.    
\end{equation*} 
Due to the boundedness of $\{u_{n}\}\subset \X$, we can find $C'>0$ such that
\begin{equation*}
\int_{\R^{N}} V(x)|u_{n}|^{2} dx \leq C' \, \mbox{ and } \, \int_{\R^{N}} |u_{n}|^{2^{*}_{\alpha}} dx \leq C', \mbox{ for all } n\in \N. 
\end{equation*}
Therefore 
\begin{equation*}
\int_{\B_{r}^{c}(0)} K(x) f(u_{n})u_{n} \, dx \leq \varepsilon C''+ C \int_{A_{n}\cap \B_{r}^{c}(0)} K(x)\, dx, 
\end{equation*} 
where $A_{n}=\{x\in \R^{N} : t_{0}\leq |u_{n}(x)|\leq t_{1}\}$. Following the same arguments in the proof of Proposition \ref{prop2.2} and by $(VK_2)$ we deduce that
\begin{equation*}
\int_{A_{n}\cap \B_{r}^{c}(0)} K(x) \, dx \rightarrow 0 \mbox{ as } r\rightarrow +\infty
\end{equation*}
uniformly in $n\in \N$ and, for $\varepsilon>0$ small enough
\begin{equation*}
\left|\int_{\B_{r}^{c}(0)} K(x) f(u_{n}) u_{n} dx \right| < (C'' + 1) \varepsilon.   
\end{equation*}
In order to complete the proof, we need to prove that
\begin{equation*}
\lim_{n\rightarrow +\infty} \int_{\B_{r}(0)} K(x) f(u_{n})u_{n} \, dx = \int_{\B_{r}(0)} K(x) f(u)u \, dx 
\end{equation*}
which easily follows by the compactness Lemma of Strauss \cite{BL1}. 
\end{proof}

\section{Technical Lemmas}\label{TL}

\noindent
In what follows we look for sign-changing weak solutions of problem (\ref{P}), that is a function $u\in \X$ such that $u^{+}:=\max \{u, 0\}\neq 0$, $u^{-}:= \min \{u, 0\}\neq 0$ in $\R^{N}$ and
\begin{equation*}
\iint_{\R^{2N}} \frac{(u(x)- u(y)) (\varphi(x)-\varphi(y))}{|x-y|^{N+2\alpha}} \, dxdy + \int_{\R^{N}} V(x) u(x) \varphi(x) \, dx = \int_{\R^{N}} K(x) f(u) \varphi(x) \, dx
\end{equation*}
for all $\varphi \in \X$. 

\noindent
In view of the assumptions on $V,K$ and $f$, we can see that the functional $J: \X\rightarrow \R$ defined by 
\begin{equation*}
J(u):= \frac{1}{2}[u]^{2} + \frac{1}{2}\int_{\R^{N}} V(x) |u|^{2} \, dx - \int_{\R^{N}} K(x) F(u) \, dx
\end{equation*} 
is Fr\'echet differentiable and that its differential $J'$ is given by
\begin{equation*}
\langle J'(u), \varphi \rangle = \iint_{\R^{2N}} \frac{(u(x)- u(y)) (\varphi(x)-\varphi(y))}{|x-y|^{N+2\alpha}} \, dx \, dy + \int_{\R^{N}} V(x) u(x) \varphi(x) \, dx - \int_{\R^{N}} K(x) f(u) \varphi(x) \, dx
\end{equation*}
for all $u, \varphi \in \X$.
Then, weak solutions to the problem (\ref{P}) are critical points of $J$.\\ 
Associated to $J$, we introduce the following Nehari manifold 
\begin{equation*}
\mathcal{N}:= \{ u\in \X \setminus \{0\} : \langle J'(u), u \rangle =0\}.
\end{equation*}
Since we look for sign-changing solutions to (\ref{P}), it is natural to seek functions $w\in \mathcal{M}$ such that
$$
J(w)=\inf_{v\in \mathcal{M}} J(v),
$$
where
\begin{equation*}
\mathcal{M}:= \{ w\in \mathcal{N} : w^{+}\neq 0, w^{-}\neq 0,  \langle J'(w), w^{+} \rangle = 0 =  \langle J'(w), w^{-} \rangle\}. 
\end{equation*}
Let us point out that for all $u\in \mathbb{X}$
$$
[u]^{2}=[u^{+}]^{2}+[u^{-}]^{2}-\iint_{\R^{2N}} \frac{(u^{+}(x)u^{-}(y)+u^{-}(x)u^{+}(y))}{|x-y|^{N+2\alpha}} \, dx \, dy\geq [u^{+}]^{2}+[u^{-}]^{2}
$$
so, we can deduce that 
\begin{align*}
J(u)=J(u^{+})+J(u^{-}) -\iint_{\R^{2N}} \frac{u^{+}(x)u^{-}(y) + u^{-}(x) u^{+}(y)}{|x-y|^{N+2\alpha}} \, dxdy,
\end{align*}
and
\begin{align*}
\langle J'(u), u^{+}\rangle =\langle J'(u^{+}), u^{+}\rangle -\iint_{\R^{2N}} \frac{u^{+}(x)u^{-}(y) + u^{-}(x) u^{+}(y)}{|x-y|^{N+2\alpha}} \, dxdy.
\end{align*}
In particular, for all $w\in \mathcal{M}$ it results
$$
\langle J'(w^{\pm}), w^{\pm} \rangle \leq 0.
$$
Motivated by \cite{AS2, BF}, we will show the existence of a minimizer of $J$ on $\mathcal{M}$ and that it is a weak solution to (\ref{P}) by using a suitable deformation argument.\\
Firstly, we collect some preliminary lemmas which will be fundamental to prove our main result.

\begin{lem}\label{lem4.1}

\noindent
\begin{compactenum}[(i)]
\item For all $u\in \mathcal{N}$ such that $\|u\|\rightarrow +\infty$, then $J(u)\rightarrow +\infty$;
\item There exists $\varrho>0$ such that $\|u\|\geq \varrho$ for all $u\in \mathcal{N}$ and $\|w^{\pm}\|\geq \varrho$ for all $w\in \mathcal{M}$. 
\end{compactenum}
\end{lem}

\begin{proof}
$(i)$ By using the definition of $\mathcal{N}$ and taking into account the assumption $(f_3)$ we get
\begin{align*}
J(u)&= J(u)-\frac{1}{\theta} \langle J'(u), u\rangle \\
&= \left( \frac{1}{2} - \frac{1}{\theta} \right) \|u\|^{2} - \int_{\R^{N}} K(x) \left[F(u)- \frac{1}{\theta} f(u)u \right] \, dx \\
&= \left( \frac{1}{2} - \frac{1}{\theta} \right) \|u\|^{2} + \frac{1}{\theta} \int_{\R^{N}} K(x) \left[f(u)u - \theta F(u) \right] \, dx \\
&\geq \left( \frac{1}{2} - \frac{1}{\theta} \right) \|u\|^{2}. 
\end{align*}
So, when $\|u\|\rightarrow +\infty$, the last inequality yields $J(u)\rightarrow +\infty$.  
\smallskip

\noindent
$(ii)$ From the assumptions $(f_1)$-$(f_2)$ or $(\tilde{f}_1)$-$(f_2)$ we deduce that, for any $\varepsilon>0$ there exists a positive constant $C_{\varepsilon}$ such that
\begin{align}
&|f(t)t| \leq \varepsilon |t|^{2} + C_{\varepsilon} |t|^{2^{*}_{\alpha}}, \mbox{ for all } t\in \R  \label{4.16} \\
&|f(t)t| \leq \varepsilon |t|^{m}+ C_{\varepsilon} |t|^{2^{*}_{\alpha}}, \mbox{ for all } t\in \R. \label{4.17}
\end{align}
Since $u\in \mathcal{N}$ we have $\langle J'(u), u \rangle =0$, that is
\begin{equation*}
\|u\|^{2} = \int_{\R^{N}} K(x) f(u)u \, dx. 
\end{equation*}
Now we distinguish two cases. \\
Assume that $(VK_3)$ holds. Then, by applying (\ref{4.16}) and Proposition \ref{prop2.2} we get 
\begin{align*}
\|u\|^{2} &\leq \left \|\frac{K}{V} \right\|_{L^{\infty}(\R^{N})} \varepsilon \int_{\R^{N}} V(x)|u|^{2} dx + C_{\varepsilon} \int_{\R^{N}} K(x) |u|^{2^{*}_{\alpha}} \, dx \\
& \leq \left \|\frac{K}{V}\right \|_{L^{\infty}(\R^{N})} \varepsilon \|u\|^{2} + C_{\varepsilon} S \|K\|_{L^{\infty}(\R^{N})} \|u\|^{2^{*}_{\alpha}}
\end{align*} 
where $S$ is the Sobolev embedding constant. Choosing $\varepsilon\in (0, 1/\|\frac{K}{V} \|_{L^{\infty}(\R^{N})} )$, we can find $\varrho_{1}>0$ such that $\|u\|\geq \varrho_{1}$. \\ 
Let us suppose that $(VK_4)$ holds true. By (\ref{4.17}) and Proposition \ref{prop2.2} we have
\begin{align}\begin{split}\label{4.19}
\|u\|^{2} &\leq \varepsilon \int_{\R^{N}} K(x)|u|^{m} dx + C_{\varepsilon} \int_{\R^{N}} K(x) |u|^{2^{*}_{\alpha}} \, dx \\ 
& \leq \varepsilon \|u\|^{m} + C_{\varepsilon} S \|K\|_{L^{\infty}(\R^{N})} \|u\|^{2^{*}_{\alpha}}. 
\end{split}\end{align} 
Since $m\in (2, 2^{*}_{\alpha})$, we can choose $\varepsilon$ sufficiently small in order to find $\varrho_{2}>0$ such that $\|u\|\geq \varrho_{2}$.\\
Now, for $w\in \mathcal{M}$, we have that $\langle J'(w), w^{\pm} \rangle =0$, and, observing that
\begin{equation*}
\iint_{\R^{2N}} \frac{w^{+}(x) w^{-}(y)+w^{-}(x) w^{+}(y)}{|x-y|^{N+2\alpha}} \, dxdy \leq 0, 
\end{equation*}
we have
\begin{equation*}
\|w^{\pm}\|^{2} \leq \int_{\R^{N}} K(x) w^{\pm} f(w^{\pm}) \, dx. 
\end{equation*}
Then we can argue as before to prove that there is $\varrho>0$ such that $\|w^{\pm}\|\geq \varrho$. 
\end{proof}

\begin{lem}\label{liminf}
Let $\{w_{n}\} \subset \mathcal{M}$ such that $w_{n}\rightharpoonup w$ in $\mathbb{X}$. Then $w^{\pm}\neq 0$.
\end{lem}

\begin{proof}
Firstly we observe that by Lemma \ref{lem4.1} there exists $\varrho>0$ such that 
\begin{equation}\label{r1}
\|w_{n}^{\pm}\|\geq \varrho \, \mbox{ for all } n\in \N. 
\end{equation}
Since $w_{n} \in \mathcal{M}$, we have $\langle J'(w_{n}), w_{n}^{\pm} \rangle =0$, that is
\begin{equation}\label{r3}
\|w_{n}^{\pm}\|^{2} - \iint_{\R^{2N}} \frac{w_{n}^{+}(x) w_{n}^{-}(y)+w_{n}^{-}(x) w_{n}^{+}(y)}{|x-y|^{N+2\alpha}} \, dxdy = \int_{\R^{N}} K(x) f(w_{n}^{\pm}) w_{n}^{\pm} \, dx. 
\end{equation}
At this point, recalling that
$$
\iint_{\R^{2N}} \frac{w_{n}^{+}(x) w_{n}^{-}(y)+w_{n}^{-}(x) w_{n}^{+}(y)}{|x-y|^{N+2\alpha}} \, dxdy\leq 0,
$$
by (\ref{r1}) and (\ref{r3}) we get
\begin{align}\label{tv}
\varrho^{2} \leq \|w_{n}^{\pm}\|^{2}\leq \int_{\R^{N}} K(x) f(w_{n}^{\pm}) w_{n}^{\pm} \, dx. 
\end{align}
Now, by using $w_{n}\rightharpoonup w$ in $\mathbb{X}$ and Proposition \ref{prop2.2}, we know that $w_{n}\rightarrow w$ in $L^{m}_{K}(\R^{N})$. Then, exploiting $|t^{\pm}-s^{\pm}|\leq |t-s|$ for all $t, s\in \R$, we can deduce that $w_{n}^{\pm}\rightarrow w^{\pm}$ in $L^{m}_{K}(\R^{N})$, and being $K(x)>0$ for all $x\in \R^{N}$, we also have $w_{n}^{\pm}\rightarrow w^{\pm}$ a.e. in $\R^{N}$.
Arguing as in Lemma \ref{lem2.1}, we can see that 
\begin{align}\label{tvb}
\int_{\R^{N}} K(x) f(w_{n}^{\pm}) w_{n}^{\pm} \, dx\rightarrow \int_{\R^{N}} K(x) f(w^{\pm}) w^{\pm} \, dx.
\end{align}
Putting together (\ref{tv}) and (\ref{tvb}), we have
$$
0<\varrho^{2} \leq \int_{\R^{N}} K(x) f(w^{\pm}) w^{\pm} \, dx
$$
which shows that $w^{\pm}\neq 0$.
\end{proof}

\begin{lem}\label{M}
If $v\in \X : v^{\pm} \neq 0$, then there exist $s, t>0$ such that
\begin{equation*}
\langle J'(tv^{+} + sv^{-}), v^{+} \rangle =0 \, \mbox{ and } \, \langle J'(tv^{+} + sv^{-}), v^{-} \rangle =0. 
\end{equation*} 
As a consequence $tv^{+} + sv^{-} \in \mathcal{M}$.  
\end{lem}

\begin{proof}
Let $G: (0, +\infty) \times (0, +\infty) \rightarrow \R^{2}$ be a continuous vector field given by 
\begin{equation*}
G(t, s) = \left( \langle J'(tv^{+} + sv^{-}), tv^{+} \rangle, \langle J'(tv^{+} + sv^{-}), sv^{-} \rangle \right)
\end{equation*} 
for every $t, s \in (0, +\infty) \times (0, +\infty)$. 
We distinguish two cases. \\
Assume that $(VK_3)$ holds. By using (\ref{4.16}) and Proposition \ref{prop2.2} we have
\begin{align}\label{piero}
\langle J'(tv^{+} + sv^{-}), tv^{+} \rangle &= t^{2} \|v^{+}\|^{2} - st \iint_{\R^{2N}} \frac{v^{+}(x) v^{-}(y)+v^{-}(x) v^{+}(y)}{|x-y|^{N+2\alpha}} \, dxdy - \int_{\R^{N}} K(x) t v^{+} f(tv^{+})\, dx \nonumber \\
&\geq t^{2} \|v^{+}\|^{2} - \int_{\R^{N}} K(x) t v^{+} f(tv^{+})\, dx \nonumber \\
&\geq t^{2} \|v^{+}\|^{2} - \varepsilon \left\|\frac{K}{V} \right\|_{L^{\infty}(\R^{N})} t^{2} \|v^{+}\|^{2} - C_{\varepsilon} t^{2^{*}_{\alpha}} \|K\|_{L^{\infty}(\R^{N})} \|v^{+}\|^{2^{*}_{\alpha}} \nonumber \\
&= \left(1- \varepsilon \left\|\frac{K}{V} \right\|_{L^{\infty}(\R^{N})}\right) t^{2} \|v^{+}\|^{2} - C_{\varepsilon} t^{2^{*}_{\alpha}} \|K\|_{L^{\infty}(\R^{N})} \|v^{+}\|^{2^{*}_{\alpha}}. 
\end{align}  
Suppose that $(VK_4)$ holds. Then by (\ref{4.17}) and Proposition \ref{prop2.2} we deduce 
\begin{align}\label{angela}
\langle J'(tv^{+} + sv^{-}), tv^{+} \rangle &= t^{2} \|v^{+}\|^{2} - st \iint_{\R^{2N}} \frac{v^{+}(x) v^{-}(y)+v^{-}(x) v^{+}(y)}{|x-y|^{N+2\alpha}} \, dxdy - \int_{\R^{N}} K(x) t v^{+} f(tv^{+})\, dx \nonumber \\
&\geq t^{2} \|v^{+}\|^{2} - \varepsilon t^{m} \|v^{+}\|^{m} - C_{\varepsilon} t^{2^{*}_{\alpha}} \|K\|_{L^{\infty}(\R^{N})} \|v^{+}\|^{2^{*}_{\alpha}}.
\end{align} 
Similarly, we can see that
\begin{align}\label{piero1}
\langle J'(tv^{+} + sv^{-}), sv^{-} \rangle \geq \left(1- \varepsilon \left\|\frac{K}{V} \right\|_{L^{\infty}(\R^{N})}\right) s^{2} \|v^{-}\|^{2} - C_{\varepsilon} s^{2^{*}_{\alpha}} \|K\|_{L^{\infty}(\R^{N})} \|v^{-}\|^{2^{*}_{\alpha}}
\end{align}
if $(VK_3)$ holds, and 
\begin{align}\label{angela1}
\langle J'(tv^{+} + sv^{-}), sv^{-} \rangle \geq s^{2} \|v^{-}\|^{2} - \varepsilon s^{m} \|v^{-}\|^{m} - C_{\varepsilon} s^{2^{*}_{\alpha}} \|K\|_{L^{\infty}(\R^{N})} \|v^{-}\|^{2^{*}_{\alpha}}
\end{align} 
under the condition $(VK_4)$. \\
Then, taking into account \eqref{piero}, \eqref{angela}, \eqref{piero1} and \eqref{angela1}, there exists $r>0$ small enough such that 
\begin{equation}\label{TWW1} 
\langle J'(rv^{+} +sv^{-}), rv^{+}\rangle >0 \mbox{ for all } s>0
\quad \mbox{ and } \quad
\langle J'(tv^{+} +rv^{-}), rv^{-}\rangle >0 \mbox{ for all } t>0.
\end{equation} 

\noindent
Taking into account $(f_3)$ and $\theta \in (2, 2^{*}_{\alpha})$, we get 
\begin{align*}
&\langle J'(tv^{+} + sv^{-}), tv^{+} \rangle \\
&= t^{2} \|v^{+}\|^{2} - st \iint_{\R^{2N}} \frac{v^{+}(x) v^{-}(y)+v^{-}(x) v^{+}(y)}{|x-y|^{N+2\alpha}} \, dxdy - \int_{\R^{N}} K(x) t v^{+} f(tv^{+})\, dx \\
&\leq t^{2} \|v^{+}\|^{2} - st \iint_{\R^{2N}} \frac{v^{+}(x) v^{-}(y)+v^{-}(x) v^{+}(y)}{|x-y|^{N+2\alpha}} \, dxdy - t^{\theta} C_1 \int_{A^{+}} K(x) |v^{+}|^{\theta}\, dx +C_{2}\|K\|_{L^{\infty}(\R^{N})} |A^{+}|
\end{align*}
and
\begin{align*}
&\langle J'(tv^{+} + sv^{-}), sv^{-} \rangle \\
&\leq s^{2} \|v^{-}\|^{2} - st \iint_{\R^{2N}} \frac{v^{+}(x) v^{-}(y)+v^{-}(x) v^{+}(y)}{|x-y|^{N+2\alpha}} \, dxdy - s^{\theta} C_1 \int_{A^{-}} K(x) |v^{-}|^{\theta}\, dx + C_{2}\|K\|_{L^{\infty}(\R^{N})} |A^{-}|
\end{align*}
where $A^{+}\subset \supp(w^{+})$ and $A^{-}\subset \supp(w^{-})$ are measurable sets with finite and positive measures.
Thus, there exists $R>0$ sufficiently large such that for all $t, s\in [r, R]$ it holds
\begin{equation}\label{TWW2}
\langle J'(Rv^{+} + sv^{-}), Rv^{+} \rangle<0 \, \mbox{ and } \, \langle J'(tv^{+} + R v^{-}), R v^{-} \rangle<0. 
\end{equation}
From \eqref{TWW1} and \eqref{TWW2}, and by applying Miranda's Theorem \cite{Miranda}, we can conclude the proof of Lemma \ref{M}. 
\end{proof}

\noindent
For each $v\in \X$ with $v^{\pm}\neq 0$, let us consider the function $h^{v}: [0, +\infty) \times [0, +\infty) \rightarrow \R$ given by
\begin{equation}\label{hv}
h^{v}(t, s) = J(tv^{+} + sv^{-})
\end{equation}
and its gradient $\Phi^{v}: [0, +\infty) \times [0, +\infty) \rightarrow \R^{2}$ defined by
\begin{align}\label{Phiv}
\Phi^{v}(t, s) &= \left ( \Phi_{1}^{v}(t, s) , \Phi_{2}^{v}(t, s) \right) \nonumber \\
&= \left(\frac{\partial h^{v}}{\partial t}(t, s), \frac{\partial h^{v}}{\partial s}(t, s)  \right) \nonumber \\
&= \left(\langle J'(tv^{+}+sv^{-}), v^{+}\rangle , \langle J'(tv^{+}+sv^{-}), v^{-}\rangle  \right). 
\end{align}
Furthermore, we consider the Jacobian matrix of $\Phi^{v}$, namely 
$$
(\Phi^{v})'(t,s)=
\begin{pmatrix}
\frac{\partial \Phi_{1}^{v}}{\partial t}(t,s) & \frac{\partial \Phi_{1}^{v}}{\partial s}(t,s)\\ 
\\
\frac{\partial \Phi_{2}^{v}}{\partial t}(t,s) & \frac{\partial \Phi_{2}^{v}}{\partial s}(t,s)
\end{pmatrix}.
$$

\noindent
In the next result we prove that, if $w\in \mathcal{M}$, the function $h^{w}$ defined in \eqref{hv} has a critical point and in particular a global maximum in $(t, s)=(1,1)$.

\begin{lem}\label{lemab} 
If $w\in \mathcal{M}$, then
\begin{compactenum}[(a)]
\item $h^{w}(t, s) < h^{w}(1,1)=J(w)$, for all $t, s\geq 0$ such that $(t,s)\neq (1,1)$; 
\item $\det(\Phi^{w})' (1,1)>0$. 
\end{compactenum}
\end{lem}

\begin{proof}
$(a)$ Since $w\in \mathcal{M}$, then $\langle J'(w), w^{\pm} \rangle =0$, that is 
\begin{align*}
&\|w^{+}\|^{2}- \iint_{\R^{2N}} \frac{w^{+}(x) w^{-}(y)+w^{-}(x) w^{+}(y)}{|x-y|^{N+2\alpha}} \, dxdy = \int_{\R^{N}} K(x) w^{+} f(w^{+})\, dx \\
&\|w^{-}\|^{2}- \iint_{\R^{2N}} \frac{w^{+}(x) w^{-}(y)+w^{-}(x) w^{+}(y)}{|x-y|^{N+2\alpha}} \, dxdy = \int_{\R^{N}} K(x) w^{-} f(w^{-})\, dx. 
\end{align*}
From this and by the definition of $\Phi^{w}$, it follows that $(1,1)$ is a critical point of $h^{w}$. \\
Thus, by using $(f_{3})$ and $F(t)\geq 0$, we have
\begin{align*}
h^{w}(t,s) &= J(tw^{+} + sw^{-}) \\
&\leq \frac{1}{2} \|tw^{+} + sw^{-} \|^{2}- \int_{A^{+}} K(x) F(tw^{+})\,dx- \int_{A^{-}} K(x) F(sw^{-})\,dx\\
&\leq \frac{t^{2}}{2} \|w^{+}\|^{2}+\frac{s^{2}}{2} \|w^{-}\|^{2}- st \iint_{\R^{2N}} \frac{w^{+}(x)w^{-}(y)+w^{+}(y)w^{-}(x)}{|x-y|^{N+2\alpha}} \,dx dy \\
&- C_{1} t^{\theta} \int_{A^{+}} K(x) |w^{+}|^{\theta} dx- C_{1} s^{\theta} \int_{A^{-}} K(x) |w^{-}|^{\theta} dx + C_{2}\|K\|_{L^{\infty}(\R^{N})} (|A^{+}|+|A^{-}|), 
\end{align*}
where $A^{+}\subset \supp(w^{+})$ and $A^{-}\subset \supp(w^{-})$ are measurable sets with finite and positive measures.
Taking into account that $\theta>2$, we can infer that   
\begin{equation*}
\lim_{|(t, s)|\rightarrow +\infty} h^{w}(t,s)= -\infty. 
\end{equation*}
By using the continuity of $h^{w}$ we can deduce the existence of $(\bar{t}, \bar{s})\in [0, +\infty) \times [0, +\infty)$ that is a global maximum point of $h^{w}$. \\
Now we prove that $\bar{t}, \bar{s}>0$. Suppose by contradiction that $\bar{s}=0$. Then $\langle J'(\bar{t}w^{+}), \bar{t}w^{+} \rangle =0$, that is 
\begin{equation}\label{4.26}
\|w^{+}\|^{2} = \int_{\R^{N}} K(x) (w^{+})^{2} \frac{f(\bar{t}w^{+})}{\bar{t}w^{+}}\, dx. 
\end{equation}
Since $w\in \mathcal{M}$, we get
\begin{align*}
\langle J'(w^{+}), w^{+} \rangle &= \langle J'(w), w^{+} \rangle + \iint_{\R^{2N}} \frac{w^{+}(x) w^{-}(y)+w^{-}(x) w^{+}(y)}{|x-y|^{N+2\alpha}} \, dxdy \\
&=\iint_{\R^{2N}} \frac{w^{+}(x) w^{-}(y)+w^{-}(x) w^{+}(y)}{|x-y|^{N+2\alpha}} \, dxdy<0
\end{align*}
which implies that
\begin{equation}\label{4.27}
\|w^{+}\|^{2} < \int_{\R^{N}} K(x) (w^{+})^{2} \frac{f(w^{+})}{w^{+}}\, dx. 
\end{equation}
Then, combining (\ref{4.26}) and (\ref{4.27}) we obtain
\begin{equation*}
0< \int_{\R^{N}} K(x) \left[ \frac{f(w^{+})}{w^{+}} - \frac{f(\bar{t}w^{+})}{\bar{t} w^{+}} \right] (w^{+})^{2} dx
\end{equation*}
that, in view of $(f_4)$, yields $\bar{t}\leq 1$. Taking into account (\ref{4.23}) we can infer
\begin{align*}
h^{w}(\bar{t}, 0)&= J(\bar{t}w^{+}) \\
&= J(\bar{t}w^{+}) - \frac{1}{2} \langle J'(\bar{t} w^{+}), \bar{t} w^{+} \rangle \\
&= \int_{\R^{N}} K(x) \left[\frac{1}{2} \bar{t} w^{+} f(\bar{t} w^{+})  - F(\bar{t} w^{+})\right]\, dx\\
&\leq \int_{\R^{N}} K(x) \left[\frac{1}{2} \bar{t} w^{+} f(\bar{t} w^{+})  - F(\bar{t} w^{+})\right]\, dx  + \int_{\R^{N}} K(x) \left[\frac{1}{2} \bar{t} w^{-} f(\bar{t} w^{-})  - F(\bar{t} w^{-})\right]\, dx \\
&=J(w) - \frac{1}{2} \langle J'(w), w \rangle \\
&=J(w) = h^{w}(1,1). 
\end{align*}
Then $h^{w}(\bar{t}, 0) < h^{w}(1,1)$, and this gives a contradiction because $(\bar{t}, 0)$ is a global maximum point. In similar fashion we can prove that $\bar{t}>0$. \\
Now we show that $\bar{s}, \bar{t}\leq1$. Since $(h^{w})'(\bar{t}, \bar{s})=0$, we get
\begin{align*}
&\bar{t}^{2}\|w^{+}\|^{2}- \bar{s}\bar{t} \iint_{\R^{2N}} \frac{w^{+}(x) w^{-}(y)+w^{-}(x) w^{+}(y)}{|x-y|^{N+2\alpha}} \, dxdy = \int_{\R^{N}} K(x) \bar{t}w^{+} f(\bar{t}w^{+})\, dx \\
&\bar{s}^{2}\|w^{-}\|^{2}- \bar{s}\bar{t} \iint_{\R^{2N}} \frac{w^{+}(x) w^{-}(y)+w^{-}(x) w^{+}(y)}{|x-y|^{N+2\alpha}} \, dxdy = \int_{\R^{N}} K(x) \bar{s} w^{-} f(\bar{s}w^{-})\, dx. 
\end{align*}
Assume that $\bar{t}\geq \bar{s}$. Since 
$$
\iint_{\R^{2N}} \frac{w^{+}(x) w^{-}(y)+w^{-}(x) w^{+}(y)}{|x-y|^{N+2\alpha}} \, dxdy \leq0
$$
we have 
\begin{align}\label{ny}
\bar{t}^{2}\|w^{+}\|^{2}- \bar{t}^{2} \iint_{\R^{2N}} \frac{w^{+}(x) w^{-}(y)+w^{-}(x) w^{+}(y)}{|x-y|^{N+2\alpha}} \, dxdy \geq \int_{\R^{N}} K(x) \bar{t}w^{+} f(\bar{t}w^{+})\, dx. 
\end{align}
Since $\langle J'(w), w^{+}\rangle =0$ ($w\in \mathcal{M}$), we deduce that
\begin{align*}
\|w^{+}\|^{2}- \iint_{\R^{2N}} \frac{w^{+}(x) w^{-}(y)+w^{-}(x) w^{+}(y)}{|x-y|^{N+2\alpha}} \, dxdy = \int_{\R^{N}} K(x) w^{+} f(w^{+})\, dx
\end{align*}
which together with (\ref{ny}) gives
\begin{equation*}
0\geq \int_{\R^{N}} K(x) \left[\frac{f(\bar{t} w^{+})}{\bar{t}w^{+}} - \frac{f(w^{+})}{w^{+}} \right]\, dx.
\end{equation*}
By $(f_4)$ we can infer that $\bar{t}\leq 1$. \\
Now we aim to prove that $h^{w}$ does not assume a global maximum in $[0,1]\times[0,1]\setminus \{(1,1)\}$, namely
\begin{equation*}
h^{w}(\bar{t}, \bar{s}) < h^{w}(1,1) \, \mbox{ for every }(\bar{t}, \bar{s})\in [0,1]\times[0,1]\setminus \{(1,1)\}.   
\end{equation*}
Let us observe that by the linearity of $F$ and the positivity of $K$ it follows that
\begin{equation*}
\int_{\R^{N}} K(x) F(w) \, dx = \int_{\R^{N}} K(x) (F(w^{+}) + F(w^{-}) ) \, dx.
\end{equation*}
Then, by the definition of $h^{w}$ and (\ref{4.23}) we get
\begin{align*}
h^{w}(\bar{t}, \bar{s}) &= J(\bar{t} w^{+} + \bar{s} w^{-}) -\frac{1}{2} \langle J'(\bar{t} w^{+} + \bar{s} w^{-}), \bar{t} w^{+}\rangle -\frac{1}{2} \langle J'(\bar{t} w^{+} + \bar{s} w^{-}), \bar{s} w^{-}\rangle \\
&=\frac{\bar{t}^{2}}{2} \|w^{+}\|^{2} + \frac{\bar{s}^{2}}{2} \|w^{-}\|^{2} - \bar{s}\bar{t} \iint_{\R^{2N}} \frac{w^{+}(x) w^{-}(y)+w^{-}(x) w^{+}(y)}{|x-y|^{N+2\alpha}} \, dxdy - \int_{\R^{N}} K(x) F(\bar{t}w^{+})\, dx \\
&- \int_{\R^{N}} K(x) F(\bar{s}w^{-})\, dx - \frac{\bar{t}^{2}}{2} \|w^{+}\|^{2}-\frac{\bar{s}^{2}}{2} \|w^{-}\|^{2} +\bar{s}\bar{t} \iint_{\R^{2N}} \frac{w^{+}(x) w^{-}(y)+w^{-}(x) w^{+}(y)}{|x-y|^{N+2\alpha}} \, dxdy \\
&+ \frac{1}{2} \int_{\R^{N}} K(x) \bar{t} w^{+} f(\bar{t} w^{+}) \, dx + \frac{1}{2} \int_{\R^{N}} K(x) \bar{s} w^{-} f(\bar{s} w^{-}) \, dx \\
&= \frac{1}{2} \int_{\R^{N}} K(x) [\bar{t} w^{+} f(\bar{t}w^{+}) - F(\bar{t} w^{+})]\, dx +\frac{1}{2} \int_{\R^{N}} K(x) [\bar{s} w^{-} f(\bar{s}w^{-}) - F(\bar{s} w^{-})]\, dx \\
&<\frac{1}{2} \int_{\R^{N}} K(x) [w^{+} f(w^{+}) - F(w^{+})]\, dx +\frac{1}{2} \int_{\R^{N}} K(x) [w^{-} f(w^{-}) - F(w^{-})]\, dx \\
&=\frac{1}{2} \int_{\R^{N}} K(x) [w f(w) - F(w)]\, dx = h^{w}(1,1).  
\end{align*}

\noindent
$(b)$ Firstly, let us observe that
\begin{align}\begin{split}\label{matrix}
&\frac{\partial \Phi_{1}^{w}}{\partial t} (t,s)= \|w^{+}\|^{2} - \int_{\R^{N}} K(x) f'(tw^{+}) (w^{+})^{2}\, dx \\
&\frac{\partial \Phi_{2}^{w}}{\partial s} (t,s)= \|w^{-}\|^{2} - \int_{\R^{N}} K(x) f'(sw^{-}) (w^{-})^{2}\, dx\\
&\frac{\partial \Phi_{1}^{w}}{\partial s} (t,s)= \frac{\partial \Phi_{2}^{w}}{\partial t} (t,s)=- \iint_{\R^{2N}} \frac{w^{+}(x)w^{-}(y)+w^{-}(x) w^{+}(y)}{|x-y|^{N+2\alpha}}\, dxdy. 
\end{split}\end{align}
Then, by using the fact that $w\in \mathcal{M}$, (\ref{matrix}) and (\ref{vince}) we have
\begin{align*}
\det(\Phi^{w})'(1,1)&= \left[\|w^{+}\|^{2} - \int_{\R^{N}} f'(w^{+})(w^{+})^{2} \, dx\right] \left[\|w^{-}\|^{2} - \int_{\R^{N}} f'(w^{-}) (w^{-})^{2} \, dx\right] \\
&- \left(\iint_{\R^{2N}} \frac{w^{+}(x)w^{-}(y)+w^{-}(x) w^{+}(y)}{|x-y|^{N+2\alpha}}\, dxdy\right)^{2} \\
&= \left[ \int_{\R^{N}} \left(w^{+} f(w^{+}) - f'(w^{+}) (w^{+})^{2}\right) \, dx + \iint_{\R^{2N}} \frac{w^{+}(x)w^{-}(y)+w^{-}(x) w^{+}(y)}{|x-y|^{N+2\alpha}}\, dxdy \right]\\
&\times \left[ \int_{\R^{N}} \left(w^{-} f(w^{-}) - f'(w^{-}) (w^{-})^{2}\right) \, dx + \iint_{\R^{2N}} \frac{w^{+}(x)w^{-}(y)+w^{-}(x) w^{+}(y)}{|x-y|^{N+2\alpha}}\, dxdy \right]\\
&- \left(\iint_{\R^{2N}} \frac{w^{+}(x)w^{-}(y)+w^{-}(x) w^{+}(y)}{|x-y|^{N+2\alpha}}\, dxdy\right)^{2}>0. 
\end{align*}
\end{proof}

\section{Proof of Theorem \ref{thm}}\label{MRT}

\noindent
In this section we will prove the existence of $w\in \mathcal{M}$ in which the infimum of $J$ is achieved on $\mathcal{M}$. Then, by using a quantitative deformation lemma, we show that $w$ is a critical point of $J$, so a sign-changing solution of (\ref{P}). 

\noindent
By using Lemma \ref{lem4.1} we know that there exists a minimizing sequence $\{w_{n}\}_{n}\subset \mathcal{M}$, bounded in $\X$, such that 
\begin{equation}\label{t00}
J(w_{n})\rightarrow \inf_{v\in \mathcal{M}} J(v)=:c_{0}>0.
\end{equation}
By using Proposition \ref{prop2.2} we can assume that 
\begin{align*}
&w_{n}^{\pm}\rightharpoonup w^{\pm}\, \mbox{ in } \X, \\
&w_{n}^{\pm}\rightarrow w^{\pm} \, \mbox{ in } L^{m}_{K}(\R^{N}), \\ 
&w_{n}^{\pm}\rightarrow w^{\pm} \, \mbox{ a.e. in } \R^{N}.
\end{align*}
From Lemma \ref{liminf} we deduce that $w^{\pm}\neq 0$, so $w= w^{+}+w^{-}$ is sign-changing. By Lemma \ref{M}, there exist $s, t>0$ such that 
\begin{equation}\label{t0}
\langle J'(tw^{+}+sw^{-}), w^{+} \rangle =0, \, \langle J'(tw^{+}+sw^{-}), w^{-} \rangle =0
\end{equation}
and $tw^{+}+sw^{-} \in \mathcal{M}$. Now, we prove that $s,t \leq 1$. Since $w_{n} \in \mathcal{M}$, we have $\langle J'(w_{n}), w_{n}^{\pm}\rangle =0$ or equivalently 
\begin{align}
&\|w_{n}^{+}\|^{2} - \iint_{\R^{2N}} \frac{w^{+}(x) w^{-}(y)+w^{-}(x) w^{+}(y)}{|x-y|^{N+2\alpha}} \, dxdy = \int_{\R^{N}} K(x) w_{n}^{+} f(w_{n}^{+})\, dx \label{a1} \\
&\|w_{n}^{-}\|^{2} - \iint_{\R^{2N}} \frac{w^{+}(x) w^{-}(y)+w^{-}(x) w^{+}(y)}{|x-y|^{N+2\alpha}} \, dxdy = \int_{\R^{N}} K(x) w_{n}^{-} f(w_{n}^{-})\, dx. \label{a2}
\end{align}
The weak lower semicontinuity of the norm $\|\cdot \|$ in $\X$ yields
\begin{equation}\label{t1}
\|w^{\pm}\|^{2}\leq \liminf_{n\rightarrow \infty} \|w_{n}^{\pm}\|^{2},
\end{equation} 
and arguing as in Lemma \ref{lem2.1}, we obtain 
\begin{equation}\label{t2}
\int_{\R^{N}} K(x) f(w_{n}^{\pm}) w_{n}^{\pm} \, dx \rightarrow \int_{\R^{N}} K(x) f(w^{\pm}) w^{\pm} \, dx. 
\end{equation}
Taking into account (\ref{a1}), (\ref{a2}), (\ref{t1}), (\ref{t2}), and by applying Fatou's lemma we deduce
\begin{align}\label{t3}
\langle J'(w), w^{+} \rangle \leq 0 \, \mbox{ and } \, \langle J'(w), w^{-} \rangle \leq 0.  
\end{align}     
Then, putting together (\ref{t0}) and (\ref{t3}), and arguing as in the proof of Lemma \ref{lemab}-$(a)$ we deduce that $s, t\leq 1$. Next, we show that $J(tw^{+}+sw^{-})=c_{0}$ and $t=s=1$. \\
By using $tw^{+}+sw^{-} \in \mathcal{M}$, $w_{n}\in \mathcal{M}$, (\ref{4.23}), (\ref{t00}) and $s, t\in (0,1]$ we can see 
\begin{align*}
c_{0}&\leq J(tw^{+} + sw^{-}) \\
&= J(tw^{+} + sw^{-})- \frac{1}{2} \langle J'(tw^{+} + sw^{-}), tw^{+} + sw^{-} \rangle \\
&=\int_{\R^{N}} K(x)\left[\frac{1}{2} f(tw^{+} + sw^{-})(tw^{+} + sw^{-})-F(tw^{+} + sw^{-})\right]dx \\
&=\int_{\R^{N}} K(x)\left[\frac{1}{2} f(tw^{+})(tw^{+})-F(tw^{+})\right]dx+\int_{\R^{N}} K(x)\left[\frac{1}{2} f(sw^{-})(sw^{-})-F(sw^{-})\right]dx \\
&\leq \int_{\R^{N}} K(x)\left[\frac{1}{2} f(w^{+})(w^{+})-F(w^{+})\right]dx+\int_{\R^{N}} K(x)\left[\frac{1}{2} f(w^{-})(w^{-})-F(w^{-})\right]dx \\
&=\lim_{n\rightarrow \infty} \left\{\int_{\R^{N}} K(x)\left[\frac{1}{2} f(w_{n}^{+})(w_{n}^{+})-F(w_{n}^{+})\right]dx+\int_{\R^{N}} K(x)\left[\frac{1}{2} f(w_{n}^{-})(w_{n}^{-})-F(w_{n}^{-})\right]dx \right\}\\
&=\lim_{n\rightarrow \infty} \int_{\R^{N}} K(x)\left[\frac{1}{2} f(w_{n})(w_{n})-F(w_{n}) \right] dx \\
&=\lim_{n\rightarrow \infty} \left[J(w_{n})-\frac{1}{2}\langle J'(w_{n}), w_{n}\rangle \right]\\
&= \lim_{n\rightarrow \infty} J(w_{n})= c_{0}.
\end{align*}        
Thus, we have proved that there exist $t, s\in (0, 1]$ such that $tw^{+} + sw^{-}\in \mathcal{M}$ and $J(tw^{+} + sw^{-})=c_{0}$. Let us observe that by the above calculation we can infer that $t=s=1$, so $w=w^{+} + w^{-}\in \mathcal{M}$ and $J(w^{+} + w^{-})=c_{0}$. \\
Finally we prove that $w$ is a critical point of $J$, that is $J'(w)=0$. We argue by contradiction. Then, we can find a positive constant $\beta>0$ and $v_{0} \in \X$ with $\|v_{0}\|=1$, such that $\langle J'(w), v_{0} \rangle =2\beta >0$. By the continuity of $J'$, we can choose a radius $R$ so that $\langle J'(v), v_{0} \rangle \geq \beta$ for every $v \in \B_{R}(w) \subset \X$ with $v^{\pm}\neq 0$. \\
Let $\mathcal{A}= (\xi, \chi)\times (\xi, \chi)\subset \R^{2}$ with $0<\xi<1<\chi$ such that
\begin{compactenum}[$(i)$]
\item $(1, 1)\in \mathcal{A}$ and $\Phi^{w}(t,s)=(0,0)$ in $\overline{\mathcal{A}}$ if and only if $(t,s)=(1,1)$, 
\item $c_{0} \notin h^{w}(\partial \mathcal{A})$,
\item $\{tw^{+} + sw^{-} : (t, s) \in \overline{\mathcal{A}}\} \subset \B_{R}(w)$
\end{compactenum}
where $h^{w}$ and $\Phi^{w}$ are defined by (\ref{hv}) and (\ref{Phiv}), and satisfy Lemma \ref{lemab}. Then we can take a radius $0<r<R$ such that 
\begin{equation}\label{t4}
\B= \overline{\B_{r}(w)} \subset \B_{R}(w) \, \mbox{ and }\, \B\cap \{tw^{+}+sw^{-} : (t, s)\in \partial \mathcal{A}\}= \emptyset.  
\end{equation}  
Now, let us define a continuous mapping $\rho: \X \rightarrow [0, +\infty)$ such that 
\begin{equation*}
\rho(u):= {\rm dist}(u, \B^{c}) \, \mbox{ for all } u\in \X,
\end{equation*}
and let $\mathbf{V}: \X \rightarrow \X$ be a bounded Lipschitz vector field given by $\mathbf{V}(u)= -\rho(u) v_{0}$. For every $u\in \X$, denoting by $\eta (\tau)= \eta (\tau, u)$, we consider the following Cauchy problem 
\begin{equation*}
\left\{
\begin{array}{ll}
\eta'(\tau)= \mathbf{V}(\eta(\tau)) \, \mbox{ for all } \tau>0, \\
\eta(0)= u. 
\end{array}
\right.
\end{equation*}
Let us note that there exist a continuous deformation $\eta(\tau, u)$ and $\tau_{0}>0$ such that for all $\tau \in [0, \tau_{0}]$ the following properties hold: 
\begin{compactenum}[$(a)$]
\item $\eta(\tau, u)=u$ for all $u\notin \B$, 
\item $\tau \rightarrow J(\eta(\tau, u))$ is decreasing for all $\eta(\tau, u)\in \B$, 
\item $J(\eta(\tau, w))\leq J(w) - \frac{r\, \beta}{2}\tau$. 
\end{compactenum} 
Indeed, $(a)$ follows by the definition of $\rho$. Regarding $(b)$, we can observe that $\langle J'(\eta(\tau)), v_{0} \rangle= \beta >0$ for $\eta(\tau)\in\B \subset \B_{R}(w)$, and, by the definition of $\rho$, we can infer $\rho(\eta(\tau))>0$. Then
\begin{equation*}
\frac{d}{d\tau}(J(\eta(\tau))) = \langle J'(\eta(\tau)), \eta'(\tau) \rangle = -\rho(\eta(\tau)) \langle J'(\eta(\tau)), v_{0} \rangle \leq-\rho(\eta(\tau))\beta<0,  \, \forall \eta(\tau)\in \B,
\end{equation*}  
that is $J(\eta(\tau, u))$ is decreasing with respect to $\tau$. \\
Now we prove $(c)$. Fix $\tau_{0}>0$ such that $\eta(\tau, u)\in \B$ for every $\tau \in [0, \tau_{0}]$, and we assume without loss of generality that
\begin{equation*}
\|\eta(\tau, w)- w\|\leq \frac{r}{2}  \Longleftrightarrow \eta(\tau, w)\in \overline{\B_{\frac{r}{2}}(w)}\, \mbox{ for any } \tau \in [0, \tau_{0}]. 
\end{equation*}
Since $\displaystyle{\rho(\eta(\tau, w))= {\rm dist}(\eta(\tau, w), \B^{c})\geq \frac{r}{2}}$, we can deduce that
\begin{equation*}
\frac{d}{d\tau} J(\eta(\tau, w))\leq -\rho(\eta(\tau, w)) \beta \leq - \frac{r\beta}{2}
\end{equation*}
and, integrating on $[0, \tau_{0}]$, we get 
\begin{equation*}
J(\eta(\tau_{0}, w))-J(w) \leq - \frac{r\beta}{2} \tau_{0}. 
\end{equation*}
Now, we consider a suitable deformed path $\bar{\eta}_{0}: \overline{\mathcal{A}} \rightarrow \X$ defined by 
\begin{equation*}
\bar{\eta}_{\tau_{0}}(t, s):= \eta (\tau_{0}, tw^{+}+sw^{-}), \, \mbox{ for all } (t,s)\in \overline{\mathcal{A}},
\end{equation*}  
and we note that
\begin{equation*}
\max_{(t,s)\in \overline{\mathcal{A}}} J(\bar{\eta}_{\tau_{0}}(t,s))<c_{0}. 
\end{equation*}
Indeed, by $(b)$ and the fact that $\eta(0,u)=u$, we have
\begin{align*}
J(\bar{\eta}_{\tau_{0}}(t,s))&= J(\eta (\tau_{0}, tw^{+}+sw^{-})) \leq  J(\eta (0, tw^{+}+sw^{-})) \\
&= J(tw^{+}+sw^{-}) = h^{w}(t,s)<c_{0} \quad \forall (t,s)\in \overline{\mathcal{A}}\setminus \{(1,1)\}, 
\end{align*}  
and for $(t,s)=(1,1)$, in virtue of $(c)$, we have
\begin{align*}
J(\bar{\eta}_{\tau_{0}}(1,1))&= J(\eta (\tau_{0}, w^{+}+w^{-}))= J(\eta(\tau_{0}, w)) \\
&\leq J(w)- \frac{r\beta}{2} \tau_{0} <J(w)=c_{0}. 
\end{align*}
Then, $\bar{\eta}_{\tau_{0}}(\overline{\mathcal{A}})\cap \mathcal{M}= \emptyset$, that is 
\begin{equation}\label{t5}
\bar{\eta}_{\tau_{0}}(t, s)\notin \mathcal{M} \, \mbox{ for all } (t,s)\in \overline{\mathcal{A}}. 
\end{equation}
On the other hand, setting $\Psi_{\tau_{0}}: \mathcal{A} \rightarrow \R^{2}$ by
\begin{equation*}
\Psi_{\tau_{0}} := \left(\frac{\langle J'(\bar{\eta}_{\tau_{0}}(t,s) ), (\bar{\eta}_{\tau_{0}}(t,s))^{+} \rangle}{t}, \frac{\langle J'(\bar{\eta}_{\tau_{0}}(t,s) ), (\bar{\eta}_{\tau_{0}}(t,s))^{-} \rangle}{s} 	\right)
\end{equation*}
we can see that, for all $(t, s)\in \partial \mathcal{A}$, by (\ref{t4}) and $(a)$ for $\tau=\tau_{0}$, it holds
\begin{equation*}
\Psi_{\tau_{0}}(t,s)= \left( \langle J'(tw^{+}+sw^{-}), w^{+} \rangle, \langle J'(tw^{+}+sw^{-}), w^{-} \rangle \right)= \Phi^{w}(t,s). 
\end{equation*} 
Then, by using Brouwer's topological degree, we have
\begin{equation*}
\rm{deg}(\Psi_{\tau_{0}}, \mathcal{A}, (0,0)) = \rm{deg}(\Phi^{w}, \mathcal{A}, (0,0))= \rm{sgn}(\det(\Phi^{w})'(1,1))=1, 
\end{equation*}
so we deduce that $\Psi_{\tau_{0}}$ has a zero $(\bar{t}, \bar{s})\in \mathcal{A}$, that is
\begin{equation*}
\Psi_{\tau_{0}}(\bar{t}, \bar{s})=(0,0) \Longleftrightarrow \langle J'(\bar{\eta}_{\tau_{0}} (\bar{t}, \bar{s})), (\bar{\eta}_{\tau_{0}} (\bar{t}, \bar{s}))^{\pm} \rangle=0. 
\end{equation*}
Therefore, there exists $(\bar{t}, \bar{s})\in \mathcal{A}$ such that $\bar{\eta}_{\tau_{0}} (\bar{t}, \bar{s}) \in \mathcal{M}$, and this is impossible in view of (\ref{t5}).

\end{document}